\PassOptionsToPackage{table,xcdraw}{xcolor} % must precede \documentclass (SIAM cls loads xcolor)
\documentclass[final,hidelinks,onefignum,onetabnum]{siamart251216}

\usepackage{amsmath}
\usepackage{amsfonts}
\usepackage{graphicx}
\usepackage{float}
\usepackage[caption=false]{subfig}
\usepackage{epstopdf}
\usepackage{algorithm}
\usepackage{algpseudocode}
\usepackage{enumitem}
\usepackage{todonotes}
\usepackage{booktabs}
\usepackage{comment}
\algrenewcommand\Require{\State \textbf{Input:} }
\algrenewcommand\Ensure{\State \textbf{Output:} }
\ifpdf
  \DeclareGraphicsExtensions{.eps,.pdf,.png,.jpg}
\else
  \DeclareGraphicsExtensions{.eps}
\fi

\newsiamremark{remark}{Remark}
\newsiamremark{hypothesis}{Hypothesis}
\crefname{hypothesis}{Hypothesis}{Hypotheses}
\newsiamthm{claim}{Claim}
\newsiamremark{fact}{Fact}
\crefname{fact}{Fact}{Facts}

\usepackage{amsopn}

\usepackage{tikz}
\usepackage{pgfplots}
\pgfplotsset{compat=1.18}
\usetikzlibrary{positioning}
\usetikzlibrary{decorations.pathreplacing}
\usetikzlibrary{shapes.geometric}

\definecolor{cbf_orange}{RGB}{240, 80, 57}
\definecolor{cbf_salmon}{RGB}{229, 122, 119}
\definecolor{cbf_pink}{RGB}{238, 186, 180}
\definecolor{cbf_darkblue}{RGB}{31, 68, 156}
\definecolor{cbf_blue}{RGB}{61, 101, 165}
\definecolor{cbf_lightblue}{RGB}{124, 161, 204}
\definecolor{cbf_gray}{RGB}{168, 182, 204}

\tikzset{
    lineg_node/.style={
        circle,
        draw=black,
        text=black,
        minimum size=0.7cm,
        inner sep=0pt,
        font=\bfseries
    }
}

\headers{Spectral Optimization of Sparse Approximate Inverses}{F. Brarda, T. Xu, V. Kalantzis, R. P. Li, and Y. Xi}

\title{Factored Sparse Approximate Inverse Preconditioning via Spectral Optimization}

\author{Francesco Brarda\thanks{Department of Mathematics, Emory University, Atlanta, GA 30307 (\email{fbrarda@emory.edu}, \email{tianshi.xu@emory.edu}, \email{yuanzhe.xi@emory.edu}). Research is supported by NSF DMS-2338904 and LLNL-LDRD program under
Project No. 24-ERD-033.}
\and Tianshi Xu\footnotemark[1]
\and Vassilis Kalantzis\thanks{IBM Research, Thomas J. Watson Research Center, 1101 Kitchawan Rd, Yorktown Heights, NY 10598 (\email{vkal@ibm.com}).}
\and Rui Peng Li\thanks{Center for Applied Scientific Computing, Lawrence Livermore National Laboratory, P. O. Box 808, L-561, Livermore, CA 94551 (\email{li50@llnl.gov}). This work was performed under the auspices of the U.S. Department of Energy by Lawrence
Livermore National Laboratory under Contract DE-AC52-07NA27344 (LLNL-JRNL-870805) and was supported by the
LLNL-LDRD program under Project No. 24-ERD-033.}
\and Yuanzhe Xi\footnotemark[1]}

\ifpdf
\hypersetup{
pdftitle={Factorized Sparse Approximate Inverse Preconditioning via Spectral Optimization},
pdfauthor={F. Brarda, T. Xu, V. Kalantzis, R. P. Li, and Y. Xi}
}
\fi

\begin{document}
\raggedbottom

\maketitle

\begin{abstract}
In this paper, we study value selection for fixed-pattern factorized sparse approximate inverse
preconditioners.  Given a prescribed sparsity pattern for a factor \(G\), we
choose its admissible entries by optimizing spectral objectives of the
congruent preconditioned operator \(P(G)=GAG^\top\).  This differs from
classical sparse approximate inverse and FSAI constructions, which choose
entries through algebraic Frobenius-residual criteria.  For symmetric positive
definite systems, the spectral target is a cluster near \(+1\).  For symmetric
indefinite systems, where congruence preserves inertia, we introduce a bimodal
loss that drives positive and negative eigenvalues toward separated clusters
near \(+1\) and \(-1\), while penalizing eigenvalues near zero.
To make these objectives practical for large sparse matrices, we derive
projected Krylov support-gradients.  Lanczos runs provide both a stochastic
trace estimate of the spectral objective and a Ritz approximation to the exact gradient.  We implement the
resulting gradient through a detached Rayleigh surrogate: the Lanczos data are
computed without gradient tracking and held fixed, while the backward pass
differentiates only recomputed Rayleigh quotients with respect to the
admissible entries of \(G\).  This avoids differentiating through the Lanczos
recurrence while returning a matrix-free gradient on the prescribed support.
We also discuss a projected Kernel Polynomial Method rule as a finite
polynomial comparison.  Experiments on finite-element test problems show that spectral value selection improves fixed-support preconditioners, especially
for symmetric indefinite saddle-point systems.
We further demonstrate a graph neural network model for predicting admissible factor entries across related matrices.
\end{abstract}

\begin{keywords}
Factorized sparse approximate inverse, stochastic
Lanczos quadrature, automatic differentiation, indefinite linear systems, spectral loss, graph neural networks
\end{keywords}

\begin{MSCcodes}
65F08, 65F10, 68T07, 68T45
\end{MSCcodes}

\section{Introduction}
\label{sec:intro}

Discretizations of partial differential equations often lead to large sparse
linear systems
\[
    Ax=b,\qquad A\in\mathbb R^{n\times n}.
\]
For large-scale problems, especially in three dimensions, Krylov subspace
methods are often preferred over direct solvers because of memory and
complexity constraints.  Their performance, however, depends strongly on the
choice of preconditioner.  Designing preconditioners that are robust,
inexpensive to apply, and suitable for parallel hardware remains a central
problem in numerical linear algebra
\cite{benzi2002preconditioning,saad2003iterative}.

Approximate inverse preconditioners are attractive in this setting because
their application consists primarily of sparse matrix-vector products, rather
than sparse triangular solves.  This makes them naturally suited to parallel
and accelerator-oriented architectures
\cite{anzt2018incomplete,xu2022pargemslr}.  Their effectiveness depends not
only on the chosen sparsity pattern, but also on the numerical values assigned
to that pattern.  This paper studies that value-selection problem for
fixed-pattern factorized sparse approximate inverse preconditioners.

In the symmetric setting, we consider preconditioners of the form
\[
    M=GG^\top,
\]
where \(G\) is a sparse factor with prescribed support
\[
    \mathcal S_G\subseteq \{1,\ldots,n\}\times\{1,\ldots,n\}.
\]
The free variables are the admissible entries
\[
    \{G_{ij}:(i,j)\in\mathcal S_G\}.
\]
For spectral analysis we use the congruent preconditioned operator
\[
    P(G)=G^\top A G .
\]
When \(G\) is nonsingular, \(P(G)\) is similar to the left-preconditioned
operator \(MA=GG^\top A\), while preserving symmetry.

Classical sparse approximate inverse preconditioners usually choose nonzero entries
through algebraic residual minimization.  SPAI minimizes
\(\|I-AM\|_F^2\), which separates into support-restricted least-squares
problems for the columns of \(M\) \cite{chow1998approximate,grote1997parallel}.  For symmetric positive definite (SPD)
matrices, FSAI constructs a sparse factor \(G\) with
\(G G^\top\approx A^{-1}\); once the factor pattern is fixed, the entries of
\(G\) are computed from small local subproblems
\cite{kolotilina1993factorized,afn}.
These criteria are computationally attractive because they avoid
eigendecomposition and lead to tractable sparse subproblems.  Their limitation
is that they are algebraic surrogates: they do not directly optimize the
spectrum of the preconditioned operator.  This mismatch can be especially
important for symmetric indefinite and saddle-point systems, where a small
Frobenius-type residual need not prevent small-magnitude eigenvalues in the
preconditioned operator, and such eigenvalues can slow Krylov convergence.

A more direct criterion is to choose the factor values according to the
spectrum of \(P(G)\).  When \(A\) is SPD, the natural target is
\[
    P(G)=G^\top AG \approx I,
\]
so the preconditioned eigenvalues should cluster near \(+1\).  For symmetric
indefinite \(A\), this one-sided target is not attainable by a nonsingular
congruence.  By Sylvester's law of inertia, \(G^\top AG\) has the same numbers
of positive and negative eigenvalues as \(A\).  The appropriate spectral
target is therefore two-sided: positive and negative eigenvalue branches should
be separated from zero and clustered near \(+1\) and \(-1\), respectively.
This motivates spectral objectives, and in particular a bimodal loss, for
choosing the admissible entries on a fixed support.

The main computational challenge is that spectral objectives depend on global
eigenvalue information from \(P(G)\).  Dense eigendecomposition can provide
this information on small matrices, but it is not scalable for large sparse
preconditioner construction.  Moreover, since the
fixed-support problem only requires derivatives with respect to entries in
\(\mathcal S_G\), the relevant task is not to form the dense gradient, but to 
approximate its restriction to the prescribed support.

Matrix-free spectral estimators reduce the cost of objective evaluation but do
not by themselves solve the gradient problem.  For example, stochastic Lanczos
quadrature (SLQ) estimates spectral traces using only matrix-vector products
\cite{ubaru2017fast}.  However, differentiating the finite SLQ computation by
automatic differentiation requires backpropagation through the full Lanczos
recurrence, including the Krylov basis, recurrence coefficients, quadrature
nodes, and quadrature weights.  This backward path can be memory intensive and
numerically sensitive, particularly when Ritz values are clustered.  We
therefore use Lanczos data in a different way: we construct a projected
approximation to the spectral derivative matrix and evaluate the resulting
gradient only on the admissible support.

A further motivation is reuse across related matrices.  Matrix-specific
spectral optimization can produce high-quality factors, but repeating this
optimization from scratch for every new matrix may be expensive.  Recent work
on learned preconditioners suggests that matrix data and problem parameters can
help transfer preconditioning information across related instances
\cite{hausner2023neural,hausner2024learning,yang2025learning,
chen2024graph,luz2020learning,li2023learning,kopanicakova2024deeponet,
huang24,huang25}.  In the fixed-pattern sparse inverse setting, however, the
prediction target has a specific structure: the outputs are admissible factor
entries \(G_{ij}\), indexed by \((i,j)\in\mathcal S_G\), rather than solution
values or features attached only to the original matrix unknowns.  This makes
amortized prediction of sparse inverse factors a distinct learning problem.

The main contributions of the paper are as follows.
\begin{enumerate}[label=\arabic*)]
\item We formulate fixed-pattern factorized sparse approximate inverse
construction as a spectral value-selection problem.  Once the support
\(\mathcal S_G\) is prescribed, the admissible entries of \(G\) are chosen by
trace objectives of the congruent operator \(P(G)=GAG^\top\), rather than by
Frobenius-type residuals.  For symmetric indefinite systems, we introduce a
bimodal loss that respects inertia by driving eigenvalues away from zero and
toward separated positive and negative clusters.

\item We develop matrix-free projected support-gradient rules for optimizing
these spectral objectives.  The main SLQ-based rule uses Lanczos Ritz values,
Ritz vectors, and quadrature weights to approximate the dense spectral
derivative, then evaluates the resulting gradient only on
\(\mathcal S_G\).  We implement this rule through a detached Rayleigh
automatic-differentiation surrogate: the Lanczos data are computed without
gradient tracking and held fixed, while the backward pass differentiates only
recomputed Rayleigh quotients.  This avoids differentiating through the
Lanczos recurrence and returns the desired local support-gradient.  We also
describe a projected Kernel Polynomial Method rule as a finite-polynomial
comparison.

\item We provide numerical evidence that spectral value selection improves
fixed-support preconditioner construction.  On symmetric indefinite systems, spectral tuning produces effective factors on supports
where Frobenius-based tuning leaves nearly singular preconditioned operators.
The experiments also show that the detached Rayleigh support-gradient is
cheaper and better aligned with the dense spectral gradient than full
automatic differentiation through finite SLQ on the tested cases.  As an amortization demonstration, we test a graph neural network factor model for predicting admissible entries across related matrices.
\end{enumerate}

The remaining sections are organized as follows.
Section~\ref{sec:background} formulates fixed-pattern sparse inverse factors
and reviews the graph representation of the sparse matrix.  Section~\ref{sec:loss} defines the spectral objectives for choosing
the factor values.  Section~\ref{sec:projected_ad} develops projected Krylov
support-gradient rules for matrix-free spectral optimization.
Section~\ref{sec:experiments} reports the numerical experiments, including
fixed-support optimization, backward-rule ablations, and the amortized
factor-prediction demonstration.  Section~\ref{sec:conclusion} concludes the
paper.
\section{Background}
\label{sec:background}

This section formalizes the fixed-support sparse inverse setting used in the
paper.  We first recall the algebraic residual viewpoint behind SPAI and FSAI,
which provides the main baseline class.  We then represent the prescribed
factor entries as a vector of free variables and introduce the graph models used to index those entries.

For a sparsity pattern \(\mathcal S\), write
\[
    \mathbb S(\mathcal S)
    =
    \{X\in\mathbb R^{n\times n}: X_{ij}=0
      \text{ whenever } (i,j)\notin\mathcal S\}.
\]
Thus the prescribed factor constraint is \(G\in\mathbb S(\mathcal S_G)\).

Classical sparse approximate inverse methods construct a sparse matrix
\(M\in\mathbb S(\mathcal S_M)\) that approximates \(A^{-1}\).  A standard
SPAI formulation is
\begin{equation}
\label{eq:spai_frobenius}
\min_{M\in\mathbb S(\mathcal S_M)}
    \|I-AM\|_F^2
=
\sum_{j=1}^n
\min_{\operatorname{supp}(m_j)\subseteq \mathcal S_j}
    \|e_j-Am_j\|_2^2,
\end{equation}
where \(m_j\) is the \(j\)th column of \(M\) and
\[
    \mathcal S_j=\{i:(i,j)\in\mathcal S_M\}.
\]
Once the column sparsity patterns are fixed, this Frobenius objective decouples
into independent support-restricted least-squares problems
\cite{grote1997parallel}.

In the factored setting, a general sparse approximate inverse may be written as
a product of sparse factors.  The symmetric case studied in this paper uses the
tied form
\[
    M=GG^\top ,
\]
where \(G\) is the stored sparse factor.  This is an orientation convention:
the same symmetric preconditioner class is often written in FSAI notation as
\(M=F^\top F\), with \(F=G^\top\).  Under the convention used here, the
symmetric congruent operator associated with the left-preconditioned matrix
\(MA=GG^\top A\) is
\[
    P(G)=G^\top A G .
\]
Indeed, if \(G\) is nonsingular, then \(MA\) is similar to \(P(G)\).

For SPD matrices, the classical FSAI method
\cite{kolotilina1993factorized} is motivated by the Cholesky factorization
\(A=LL^\top\).  In the common orientation \(M=F^\top F\), FSAI seeks a sparse
inverse Cholesky factor \(F\approx L^{-1}\), often through the objective
\[
    \min_{F\in\mathbb S(\mathcal S_F)} \|I-FL\|_F^2 .
\]
The Cholesky factor \(L\) need not be formed explicitly; for a prescribed
pattern, the entries of the inverse factor can be computed from small local
subproblems involving principal submatrices of \(A\).  

A related nonsymmetric factorized approximate inverse construction is AINV
\cite{benzi1998sparse}.  For a nonsymmetric matrix \(A\), AINV seeks sparse
factors \(Z\) and \(W\), together with a diagonal matrix \(D\), such that
\[
    W^\top A Z \approx D .
\]
Equivalently, the columns of \(Z=[z_1,\ldots,z_n]\) and
\(W=[w_1,\ldots,w_n]\) are constructed to be approximately
\(A\)-biorthogonal:
\[
    w_i^\top A z_j \approx 0 \quad (i\neq j),
    \qquad
    d_i = w_i^\top A z_i .
\]
In practice, \(Z\) and \(W\) are generated by an incomplete biconjugation
sweep.  Starting from coordinate vectors, previously computed directions are
used to eliminate the current \(A\)-couplings, while small entries are dropped
to control fill.  Thus the updates have the form
\[
    z_j \leftarrow \operatorname{drop}
    \left(z_j-\frac{w_i^\top A z_j}{d_i}z_i\right),
    \qquad
    w_j \leftarrow \operatorname{drop}
    \left(w_j-\frac{w_j^\top A z_i}{d_i}w_i\right),
    \qquad j>i .
\]
AINV therefore provides a sequential algebraic sparse-inverse factorization,
with sparsity usually controlled by dropping tolerances or fill restrictions.
This contrasts with the fixed-support least-squares construction in
\eqref{eq:spai_frobenius}, while sharing the same goal of producing an
explicit sparse approximation to \(A^{-1}\).

For the fixed-support formulation used here, let \(N_G=|\mathcal S_G|\).
The admissible entries of \(G\) are collected in a vector
\(g\in\mathbb R^{N_G}\) by
\[
    G(g)=\sum_{(i,j)\in\mathcal S_G} g_{ij}E_{ij},
\]
where \(E_{ij}\) has a single nonzero entry 1 at position \((i,j)\).  Entries
outside \(\mathcal S_G\) are fixed to zero.  The same variables \(g\) are used
later for matrix-specific optimization and for parameterized prediction by
\(G_\theta(A,\beta)\).  The spectral objectives for choosing these values are
introduced in Section~\ref{sec:loss}.

To describe the graph-based architecture used later to predict entry values, we also introduce the graph representation associated with a sparse matrix.
Such representations are standard in sparse matrix computations
\cite{parter1961use}.  We use the usual adjacency graph view of the original sparse matrix pattern,
\[
    \mathcal G=(\mathcal V,\mathcal E),
    \qquad
    \mathcal V=\{1,\ldots,n\}.
\]
The node \(i\in\mathcal V\) corresponds to the unknown \(x_i\) in \(Ax=b\).
For \(i\ne j\), an edge \(\{i,j\}\in\mathcal E\) is present when
\(A_{ij}\) or \(A_{ji}\) belongs to the sparsity pattern of \(A\).  
\section{Spectral Objectives for Sparse Inverse Factor Optimization}
\label{sec:loss}

This section defines the spectral objectives used to choose the numerical
values of the admissible entries of \(G\).  We use the convention
\[
    M=GG^\top,
    \qquad
    P(G)=G^\top A G ,
\]
so that, when \(G\) is nonsingular, the symmetric congruent operator \(P(G)\)
is similar to the left-preconditioned operator \(MA=GG^\top A\).  The main
focus is the indefinite case, where the preconditioned spectrum should be
two-sided and separated away from zero.

\subsection{Spectral targets on a fixed support}
\label{sec:loss:targets}

Classical sparse approximate inverse methods usually choose inverse or factor
entries by algebraic residual criteria, such as the Frobenius objective in
\eqref{eq:spai_frobenius}.  
Despite their practical advantages, these criteria do not control the spectrum of the preconditioned operator.
In the fixed-support setting, this distinction is important: the admissible
entries are the same for all methods, and only the numerical values assigned
to those entries are changed.

Let
\[
    \lambda_1(P(G)),\ldots,\lambda_n(P(G))
\]
denote the eigenvalues of
\[
    P(G)=G^\top A G .
\]
Consider trace
objectives
\begin{equation}
\label{eq:exact_spectral_loss}
    \mathcal L_f(G)
    =
    \frac{1}{n}\sum_{i=1}^n f(\lambda_i(P(G)))
    =
    \frac{1}{n}\operatorname{tr}(f(P(G))),
\end{equation}
where \(f:\mathbb R\to\mathbb R\) is a scalar function that specifies the desired spectral behavior of the preconditioned operator.
Although these objectives do not minimize the Krylov
iteration count directly, they penalize spectral features that are unfavorable
for Krylov convergence.

For SPD systems, the natural spectral target is a single cluster near \(+1\). 
Indeed, if \(A\) is SPD and \(G\) is nonsingular, then \(P(G)\) is also SPD. Consequently, clustering the eigenvalues of \(P(G)\) near \(+1\) yields a favorable condition number for the preconditioned system. 
A simple penalty function that promotes this behavior is 
\[
    f_+(\lambda)=(\lambda-1)^2 .
\]
The corresponding trace objective becomes
\[
    \mathcal L_{f_+}(G)
    =
    \frac{1}{n}\operatorname{tr}\!\left((P(G)-I)^2\right)
    =
    \frac{1}{n}\|P(G)-I\|_F^2,
\]
and its minimization induces the ideal relation
\[
    G^\top A G \approx I .
\]
Although this penalty is smooth and simple, the resulting optimization problem over the factor entries is generally nonconvex.

For symmetric indefinite systems, the spectral target must be two-sided.  If
\(G\) is nonsingular, then \(P(G)=G^\top A G\) is congruent to \(A\).  By
Sylvester's law of inertia, \(P(G)\) has the same numbers of positive and
negative eigenvalues as \(A\).  Therefore an indefinite operator cannot be
transformed by this preconditioner into a matrix whose eigenvalues all cluster
near \(+1\).  The corresponding ideal target is instead that positive
eigenvalues cluster near \(+1\) and negative eigenvalues cluster near \(-1\),
while preserving the inertia of \(A.\)  With a prescribed sparse factor pattern,
exact two-point clustering is generally not attainable, but it is still
desirable to move eigenvalues away from zero and toward separated positive and
negative clusters.

We therefore use the bimodal penalty function
\begin{equation}
\label{eq:bimodal_loss}
    f_{\pm}(\lambda)
    =
    (|\lambda|_\delta-1)^2
    +
    \alpha \bigl(\log(|\lambda|_\delta+\varepsilon)\bigr)^2,
\end{equation}
where
\[
    |\lambda|_\delta=(\lambda^2+\delta^2)^{1/2},
    \qquad
    \alpha>0,\quad \varepsilon>0 .
\]

The first term in \eqref{eq:bimodal_loss} is the two-sided clustering term: it
encourages \(|\lambda|\approx 1\), and hence favors clusters near \(+1\) and
\(-1\).  This term already penalizes eigenvalues near zero, but only by a
bounded amount.  For example, without smoothing,
\((|\lambda|-1)^2\) assigns the same penalty to \(\lambda=0\) and
\(\lambda=2\), although an eigenvalue near zero is much more harmful for
Krylov convergence.  The logarithmic term is included as a soft anti-zero
penalty.  It increases the relative cost of small \(|\lambda|\) and
discourages near-null directions in the preconditioned operator.  The
parameter \(\alpha\) controls the strength of this correction, \(\delta\)
smooths the absolute value, and \(\varepsilon\) prevents a singular logarithm.

For differentiability of the loss at the origin, we take \(\delta>0\) in the
analysis.  Defining
\[
    s_\delta(\lambda)=(\lambda^2+\delta^2)^{1/2},
\]
the derivative of the bimodal penalty is
\[
    f_\pm'(\lambda)
    =
    \frac{2\lambda}{s_\delta(\lambda)}
    \left[
        s_\delta(\lambda)-1
        +
        \alpha
        \frac{\log(s_\delta(\lambda)+\varepsilon)}
             {s_\delta(\lambda)+\varepsilon}
    \right].
\]

Figure~\ref{fig:bimodal_loss_demo} illustrates the role of the logarithmic
term.  The clustering term alone already has preferred locations near
\(\pm1\), but the full bimodal loss raises the penalty near zero while
preserving the two-sided structure of the target.  This emphasizes the part of
the spectrum that is most unfavorable for Krylov convergence.  As with the SPD
objective, the resulting factor optimization problem remains nonconvex.

\begin{figure}[t]
\centering
\definecolor{clusteronly}{RGB}{31,119,180}
\definecolor{fullbimodal}{RGB}{217,95,2}

% Visualization parameters
\pgfmathsetmacro{\deltavis}{0.02}
\pgfmathsetmacro{\alphavis}{0.1}
\pgfmathsetmacro{\epsvis}{0.001}

% Exact values at lambda = 0 for the displayed parameters
\pgfmathsetmacro{\szerovis}{sqrt(\deltavis^2)}
\pgfmathsetmacro{\clusterzerovis}{(\szerovis-1)^2}
\pgfmathsetmacro{\fullzerovis}{(\szerovis-1)^2
    + \alphavis*(ln(\szerovis+\epsvis))^2}

\begin{tikzpicture}
\begin{axis}[
    width=0.78\linewidth,
    height=0.42\linewidth,
    xlabel={preconditioned eigenvalue $\lambda$},
    ylabel={penalty},
    xmin=-2.5,
    xmax=2.5,
    ymin=0,
    ymax=4.2,
    domain=-2.5:2.5,
    samples=400,
    xtick={-2,-1,0,1,2},
    legend style={
        font=\scriptsize,
        at={(0.5,1.04)},
        anchor=south,
        legend columns=2
    },
    grid=both,
    grid style={line width=.1pt, draw=gray!25},
    major grid style={line width=.2pt, draw=gray!35},
]

% Two-sided clustering term only
\addplot[thick, dashed, color=clusteronly]
    {(sqrt(x^2+\deltavis^2)-1)^2};
\addlegendentry{$(|\lambda|_\delta-1)^2$}

% Full bimodal loss
\addplot[thick, color=fullbimodal]
    {(sqrt(x^2+\deltavis^2)-1)^2
      + \alphavis*(ln(sqrt(x^2+\deltavis^2)+\epsvis))^2};
\addlegendentry{$f_{\pm}(\lambda)$}

% Reference lines
\addplot[densely dotted, color=gray!65] coordinates {(-1,0) (-1,4.2)};
\addplot[densely dotted, color=gray!65] coordinates {(0,0) (0,4.2)};
\addplot[densely dotted, color=gray!65] coordinates {(1,0) (1,4.2)};

% Reference labels
\node[font=\scriptsize, anchor=south] at (axis cs:-1,3.75) {$-1$};
\node[font=\scriptsize, anchor=south] at (axis cs:0,3.75) {$0$};
\node[font=\scriptsize, anchor=south] at (axis cs:1,3.75) {$+1$};

% Formula-driven arrow showing the added penalty at zero
\draw[<->, thick, color=fullbimodal]
    (axis cs:0,\clusterzerovis) -- (axis cs:0,\fullzerovis);
\node[font=\scriptsize, align=center, anchor=west, color=fullbimodal]
    at (axis cs:0.08,1.7)
    {extra penalty\\near zero};

\node[font=\scriptsize, align=center]
    at (axis cs:-1.65,0.35)
    {preferred\\near $-1$};
\node[font=\scriptsize, align=center]
    at (axis cs:1.65,0.35)
    {preferred\\near $+1$};

\end{axis}
\end{tikzpicture}
\caption{Illustration of the bimodal loss with
\(\delta=0.02\), \(\alpha=0.1\), and \(\varepsilon=10^{-3}\).  The dashed
blue curve shows the two-sided clustering term
\((|\lambda|_\delta-1)^2\), which favors eigenvalues near \(\pm1\) and
penalizes zero by a bounded amount.  The solid orange curve adds the
logarithmic term in \eqref{eq:bimodal_loss}, increasing the relative cost of
small-magnitude eigenvalues while preserving the same two-well structure.}
\label{fig:bimodal_loss_demo}
\end{figure}
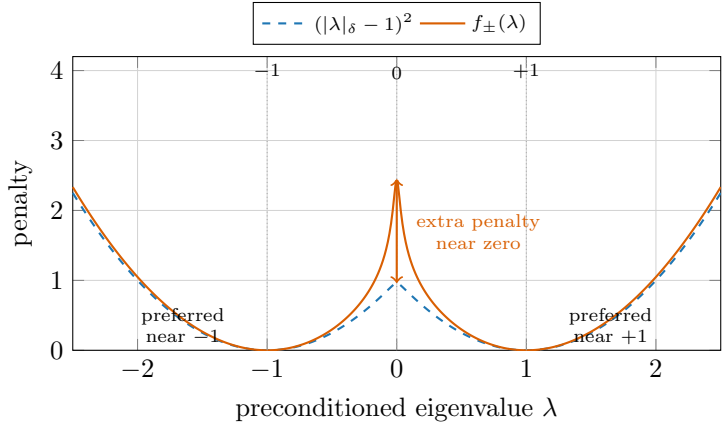

\subsection{Exact spectral gradient and support restriction}
\label{sec:loss:gradient}

We next derive the exact gradient of the spectral trace objective.  The result
is useful for two reasons.  First, it identifies the dense spectral quantity
that must be approximated in a scalable method.  Second, it shows that, in the
fixed-support problem, only the entries of this gradient on
\(\mathcal S_G\) are needed.

\begin{proposition}
\label{prop:spectral_gradient}
Let \(A=A^\top\), \(P(G)=G^\top A G\), and
\[
    \mathcal L_f(G)=\frac{1}{n}\operatorname{tr}(f(P(G))).
\]
Assume that \(f\in C^1\) on an open interval containing the spectrum of
\(P(G)\).  Then
\[
    \nabla_G \mathcal L_f(G)
    =
    \frac{2}{n} A G f'(P(G)).
\]
If \(G\) is restricted to the prescribed support \(\mathcal S_G\), then the
gradient with respect to the free entries is
\[
    \frac{\partial \mathcal L_f}{\partial G_{ij}}
    =
    \left(
    \frac{2}{n} A G f'(P(G))
    \right)_{ij},
    \qquad (i,j)\in\mathcal S_G .
\]
\end{proposition}

\begin{proof}
Let
\[
    F=f'(P(G)).
\]
Since \(P(G)\) is symmetric and \(f\) is differentiable on an interval
containing its spectrum,
\[
    d\,\operatorname{tr}(f(P))=\operatorname{tr}(F\,dP).
\]
For
\[
    P(G)=G^\top A G,
\]
we have
\[
    dP=dG^\top A G+G^\top A\,dG .
\]
Therefore
\[
\begin{aligned}
    d\mathcal L_f
    &=
    \frac{1}{n}
    \operatorname{tr}\!\left(
        F(dG^\top A G+G^\top A\,dG)
    \right).
\end{aligned}
\]
For the first term,
\[
    \operatorname{tr}(F dG^\top A G)
    =
    \operatorname{tr}(dG^\top A G F)
    =
    \operatorname{tr}\!\left((A G F)^\top dG\right).
\]
For the second term, using the symmetry of \(A\) and \(F\),
\[
    \operatorname{tr}(F G^\top A\,dG)
    =
    \operatorname{tr}\!\left((A G F)^\top dG\right).
\]
Hence
\[
    d\mathcal L_f
    =
    \frac{2}{n}
    \operatorname{tr}\!\left((A G F)^\top dG\right),
\]
from which the full gradient follows as
\[
    \nabla_G \mathcal L_f(G)
    =
    \frac{2}{n} A G F
    =
    \frac{2}{n} A G f'(P(G)).
\]
If \(G\) is represented by independent entries on \(\mathcal S_G\), then
admissible perturbations \(dG\) are supported only on \(\mathcal S_G\).  Thus
only the corresponding entries of the full matrix gradient are active.
\end{proof}

Proposition~\ref{prop:spectral_gradient} identifies the computational
bottleneck in the dense spectral gradient:
\[
    \nabla_G\mathcal L_f(G)=2A G H(G),
    \qquad
    H(G)=\frac1n f'(P(G)).
\]
The matrix \(H(G)\) is generally dense and requires global spectral information
to form.  The next section develops matrix-free Krylov approximations to this
spectral-gradient structure and evaluates the resulting gradient only on the
prescribed support.
\section{Projected Krylov Differentiation}
\label{sec:projected_ad}
In this section, we develop two Krylov-projected backward rules for computing
matrix-free gradients on the prescribed factor support.  The main SLQ-based
rule uses Lanczos Ritz data to approximate the spectral derivative matrix
\(H(G)=n^{-1}f'(P(G))\), and then evaluates the resulting gradient only on
\(\mathcal S_G\) through a detached Rayleigh surrogate.  We also describe a
projected KPM rule that differentiates a finite Chebyshev polynomial trace
estimator, providing a finite-polynomial comparison to the Lanczos--Rayleigh
construction.

\subsection{Lanczos value and derivative estimates}
\label{sec:projected_ad:slq}

Let \(G_0\) be the current factor and set
\[
    P_0=P(G_0)=G_0^\top A G_0,
    \qquad
    H_0=\frac1n f'(P_0).
\]
At \(G_0\), an optimizer needs two quantities: a scalar objective value and a
gradient with respect to the admissible entries of \(G\).  Lanczos provides
both from the same Krylov runs.  The scalar SLQ estimator approximates
\(\mathcal L_f(G_0)\), while the same Ritz data give a projected approximation
to \(H_0\), the spectral matrix appearing in the exact gradient formula
\[
    \nabla_G\mathcal L_f(G_0)=2A G_0 H_0 .
\]

We now define the Lanczos quantities from which both the value estimate and
the projected derivative approximation are built.  Let
\(z_\ell\), \(\ell=1,\ldots,N_{\mathrm p}\), be random probe vectors satisfying
\[
    \mathbb E[z_\ell z_\ell^\top]=I .
\]
For each probe, we run \(m\) Lanczos steps on \(P_0\), starting from the normalized vector
\[
    q_{\ell,1}=z_\ell/\|z_\ell\|_2 .
\]
This yields to
\[
    P_0Q_\ell
    =
    Q_\ell T_\ell
    +
    \beta_{\ell,m}q_{\ell,m+1}e_m^\top,
    \qquad
    Q_\ell^\top Q_\ell=I .
\]
We then compute the eigendecomposition of $T_\ell,$ defining 
\[
    T_\ell y_{\ell j}=\theta_{\ell j}y_{\ell j},
    \qquad
    v_{\ell j}=Q_\ell y_{\ell j},
    \qquad
    \omega_{\ell j}=(e_1^\top y_{\ell j})^2 ,
\]
where \(\theta_{\ell j}\) are the Ritz values, \(v_{\ell j}\) the Ritz vectors, and \(\omega_{\ell j}\) the usual SLQ weights.

The value estimate is the standard SLQ trace estimator
\begin{equation}
\label{eq:slq_estimator}
    \widehat{\mathcal L}_{m,N_{\mathrm p}}(G_0)
    =
    \frac{1}{nN_{\mathrm p}}
    \sum_{\ell=1}^{N_{\mathrm p}}
    \|z_\ell\|_2^2
    \sum_{j=1}^{m}
    \omega_{\ell j} f(\theta_{\ell j}) .
\end{equation}
This scalar is used as the matrix-free objective value at \(G_0\), for example
in line search, monitoring, or comparing candidate factors.

For the gradient, we reuse the same Ritz values, Ritz vectors, and weights,
but apply them to \(f'\).  To make this connection explicit, suppose that
\[
    P_0=\sum_{r=1}^n \lambda_r u_r u_r^\top .
\]
Then
\[
    H(G_0)
    =
    \frac1n f'(P_0)
    =
    \frac1n\sum_{r=1}^n f'(\lambda_r)u_r u_r^\top .
\]
For an isotropic probe \(z\),
\[
    \mathbb E[(z^\top u_r)^2]
    =
    u_r^\top\mathbb E[zz^\top]u_r
    =
    1 .
\]
Thus \(H(G_0)\) can be written as an expectation of probe-weighted spectral
expansions:
\[
\begin{aligned}
    \mathbb E_z
    \left[
    \frac1n
    \sum_{r=1}^n
    (z^\top u_r)^2 f'(\lambda_r)u_r u_r^\top
    \right]
    &=
    \frac1n
    \sum_{r=1}^n
    \mathbb E_z\!\left[(z^\top u_r)^2\right]
    f'(\lambda_r)u_r u_r^\top  \\
    &=
    \frac1n
    \sum_{r=1}^n
    f'(\lambda_r)u_r u_r^\top
    =
    H(G_0).
\end{aligned}
\]
A finite number of probes replaces this expectation by an empirical average,
and Lanczos replaces the exact eigenpairs
\((\lambda_r,u_r)\) by Ritz pairs \((\theta_{\ell j},v_{\ell j})\).
Moreover, since \(q_{\ell,1}=z_\ell/\|z_\ell\|_2\) and
\(v_{\ell j}=Q_\ell y_{\ell j}\),
\[
    (z_\ell^\top v_{\ell j})^2
    =
    \|z_\ell\|_2^2(e_1^\top y_{\ell j})^2
    =
    \|z_\ell\|_2^2\omega_{\ell j}.
\]
This motivates the matrix-valued Ritz approximation
\begin{equation}
\label{eq:slq_projected_measure}
    \widehat H(G_0)
    =
    \frac{1}{nN_{\mathrm p}}
    \sum_{\ell=1}^{N_{\mathrm p}}
    \|z_\ell\|_2^2
    \sum_{j=1}^{m}
    \omega_{\ell j}
    f'(\theta_{\ell j})
    v_{\ell j}v_{\ell j}^\top .
\end{equation}
The value estimate \eqref{eq:slq_estimator} and the derivative approximation
\eqref{eq:slq_projected_measure} therefore reuse the same probe vectors,
Lanczos bases, Ritz values, Ritz vectors, and quadrature weights.  The value
estimate uses \(f(\theta_{\ell j})\), whereas the gradient approximation uses
\(f'(\theta_{\ell j})v_{\ell j}v_{\ell j}^\top\).

Substituting \(\widehat H(G_0)\) for \(H(G_0)\) in the dense gradient formula
gives
\[
    \nabla_G\mathcal L_f(G_0)
    =
    2A G_0 H(G_0)
    \;\approx\;
    2A G_0\widehat H(G_0).
\]
The dense matrix \(\widehat H(G_0)\) is not formed.  For an admissible entry
\((i,k)\in\mathcal S_G\), the corresponding projected gradient entry is
\begin{equation}
\label{eq:ritz_support_gradient}
    \widehat g_{ik}(G_0)
    =
    \frac{2}{nN_{\mathrm p}}
    \sum_{\ell=1}^{N_{\mathrm p}}
    \|z_\ell\|_2^2
    \sum_{j=1}^{m}
    \omega_{\ell j}
    f'(\theta_{\ell j})
    \bigl(A G_0 v_{\ell j}\bigr)_i
    v_{\ell j,k}.
\end{equation}
This formula is the restriction of \(2A G_0\widehat H(G_0)\) to
\(\mathcal S_G\).  It uses sparse products with \(G_0\) and \(A\), followed by
gathers on the prescribed support.

The approximation error is controlled by the accuracy of
\(\widehat H(G_0)\) as an approximation to \(H(G_0)\).  Let
\(\Pi_{\mathcal S_G}\) denote the Frobenius-orthogonal projection onto the
prescribed support, and define
\[
    g_{\mathcal S}^{\rm exact}
    =
    \Pi_{\mathcal S_G}\!\left(2A G_0H(G_0)\right),
    \qquad
    g_{\mathcal S}^{\rm Ritz}
    =
    \Pi_{\mathcal S_G}\!\left(2A G_0\widehat H(G_0)\right).
\]
Then
\[
    g_{\mathcal S}^{\rm Ritz}-g_{\mathcal S}^{\rm exact}
    =
    2\Pi_{\mathcal S_G}
    \left(A G_0(\widehat H(G_0)-H(G_0))\right),
\]
and since \(\Pi_{\mathcal S_G}\) is nonexpansive in the Frobenius norm,
\[
    \|g_{\mathcal S}^{\rm Ritz}-g_{\mathcal S}^{\rm exact}\|_F
    \le
    2\|A G_0\|_F\,\|\widehat H(G_0)-H(G_0)\|_2 .
\]
Thus the support-gradient error is reduced to the matrix-function
approximation error \(\widehat H(G_0)-H(G_0)\).  This error has two sources:
finite-probe sampling error and Lanczos/Ritz approximation error.  In
practice, the projected gradient is most useful when the Ritz vectors capture
the spectral components emphasized by \(f'\), especially the near-zero
components emphasized by the bimodal indefinite loss.

The next subsection shows how the same value-gradient pair can be implemented
by differentiating a detached Rayleigh surrogate: the forward pass recovers
the scalar estimate \eqref{eq:slq_estimator} at \(G_0\), and the backward pass
returns the support-gradient \eqref{eq:ritz_support_gradient} without
differentiating through the Lanczos recurrence.

\subsection{Detached Rayleigh implementation of the support-gradient}
\label{sec:projected_ad:rayleigh}

The previous subsection derived the projected support-gradient
\eqref{eq:ritz_support_gradient}.  This formula can be implemented directly
using sparse products with \(G_0\) and \(A\), followed by gathers on
\(\mathcal S_G\).  Such an implementation is mathematically straightforward
but requires careful bookkeeping over the support and over all Ritz vectors.
We now show that the same local support-gradient is obtained by differentiating
a detached Rayleigh surrogate.

At the current factor \(G_0\), compute the Lanczos data
\[
    \{v_{\ell j},\omega_{\ell j},\theta_{\ell j}\}.
\]
In the local backward calculation, the Ritz vectors and quadrature weights are
held fixed.  For a variable factor \(G\), define the Rayleigh quotients
\[
    \rho_{\ell j}(G)
    =
    v_{\ell j}^{\top}P(G)v_{\ell j}
    =
    v_{\ell j}^{\top}G^\top A G v_{\ell j}.
\]
The detached Rayleigh surrogate is
\begin{equation}
\label{eq:detached_rayleigh_surrogate}
    \widetilde{\mathcal L}_{m,N_{\mathrm p}}(G;G_0)
    =
    \frac{1}{nN_{\mathrm p}}
    \sum_{\ell=1}^{N_{\mathrm p}}
    \|z_\ell\|_2^2
    \sum_{j=1}^{m}
    \omega_{\ell j}
    f\!\left(\rho_{\ell j}(G)\right).
\end{equation}
Only the Rayleigh quotients depend on \(G\).  The Lanczos computation is
repeated at each new forward evaluation, so the Ritz data track the current
factor, but they are fixed for the local backward pass.

\begin{proposition}[Detached Rayleigh support-gradient]
\label{prop:detached_rayleigh_gradient}
Assume \(A=A^\top\) and suppose \(f\) is differentiable at the Rayleigh
quotients \(\rho_{\ell j}(G)\).  With \(v_{\ell j}\) and \(\omega_{\ell j}\)
held fixed,
\[
    \nabla_G
    \widetilde{\mathcal L}_{m,N_{\mathrm p}}(G;G_0)
    =
    2A G\widehat H(G;G_0),
\]
where
\[
    \widehat H(G;G_0)
    =
    \frac{1}{nN_{\mathrm p}}
    \sum_{\ell=1}^{N_{\mathrm p}}
    \|z_\ell\|_2^2
    \sum_{j=1}^{m}
    \omega_{\ell j}
    f'\!\left(\rho_{\ell j}(G)\right)
    v_{\ell j}v_{\ell j}^{\top}.
\]
At \(G=G_0\), in exact arithmetic,
\[
    \rho_{\ell j}(G_0)=\theta_{\ell j},
    \qquad
    \widehat H(G_0;G_0)=\widehat H(G_0).
\]
Consequently, the restriction of this gradient to \(\mathcal S_G\) agrees
with the Lanczos--Ritz support-gradient in \eqref{eq:ritz_support_gradient}.
\end{proposition}

\begin{proof}
For a fixed vector \(v\), let
\[
    \rho(G)=v^\top G^\top A G v .
\]
Since \(A=A^\top\),
\[
\begin{aligned}
    d\rho
    &=
    v^\top dG^\top A G v
    +
    v^\top G^\top A\,dG\,v  \\
    &=
    2\,\operatorname{tr}\!\left((A G vv^\top)^\top dG\right).
\end{aligned}
\]
Applying the scalar chain rule to each term in
\eqref{eq:detached_rayleigh_surrogate}, with \(v_{\ell j}\) and
\(\omega_{\ell j}\) fixed, gives
\[
    \nabla_G
    \widetilde{\mathcal L}_{m,N_{\mathrm p}}(G;G_0)
    =
    2A G\widehat H(G;G_0).
\]
At \(G=G_0\), the Ritz relation gives
\[
    v_{\ell j}^{\top}P(G_0)v_{\ell j}=\theta_{\ell j}
\]
in exact arithmetic, so \(\widehat H(G_0;G_0)=\widehat H(G_0)\).  Restricting
the resulting matrix gradient to \(\mathcal S_G\) gives
\eqref{eq:ritz_support_gradient}.
\end{proof}

Thus the detached Rayleigh surrogate is not a different projected-gradient
estimator.  It is an implementation of the same local support-gradient
\eqref{eq:ritz_support_gradient}.  At the current factor, its forward value
agrees with the SLQ value estimate,
\[
    \widetilde{\mathcal L}_{m,N_{\mathrm p}}(G_0;G_0)
    =
    \widehat{\mathcal L}_{m,N_{\mathrm p}}(G_0),
\]
while its backward pass returns the projected support-gradient.  Therefore a
detached automatic-differentiation implementation evaluates one scalar
surrogate: the forward value can be used by a line-search method or for
diagnostics, and the backward pass gives the gradient used to update the
admissible entries of \(G\).  If \(G\) is represented directly by its free
entries on \(\mathcal S_G\), automatic differentiation returns a
support-indexed gradient and does not form the dense matrix \(\widehat H\) or
the dense product \(AG\widehat H\).

This rule should be distinguished from full automatic differentiation of the
finite SLQ estimator \eqref{eq:slq_estimator}.  Full SLQ automatic
differentiation differentiates through the Lanczos recurrence, including the
Lanczos basis, Ritz values, Ritz vectors, and quadrature weights, with respect
to \(G\).  By contrast, the detached rule uses the Lanczos pass only to define
local spectral data; it then holds these data fixed and differentiates the
recomputed Rayleigh quotients.  The resulting gradient is therefore biased
relative to the pathwise derivative of the finite SLQ estimator, but it is
designed to approximate the support-restricted gradient of the underlying
spectral trace objective.  The numerical experiments in
\Cref{sec:exp:backward_ablation} show that, on the tested problems, this
detached Rayleigh gradient is both cheaper and better aligned with the dense
support-restricted spectral gradient than full SLQ automatic differentiation.

Figure~\ref{fig:direct_vs_detached_lanczos} summarizes this distinction.  In
the detached implementation, the only differentiated quantities are the
Rayleigh quotients \(v_{\ell j}^{\top}P(G)v_{\ell j}\), and the output is the
gradient on the prescribed support \(\mathcal S_G\).

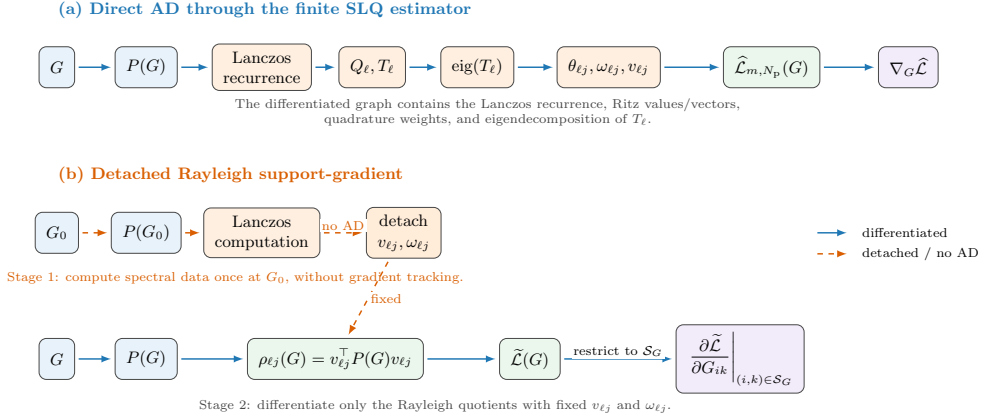
\begin{figure}[t]
\centering
\definecolor{diffblue}{RGB}{31,119,180}
\definecolor{detachorange}{RGB}{217,95,2}
\definecolor{boxgray}{RGB}{245,245,245}
\definecolor{lightblue}{RGB}{232,242,250}
\definecolor{lightorange}{RGB}{254,239,224}
\definecolor{lightgreen}{RGB}{235,247,236}
\definecolor{lightpurple}{RGB}{242,235,250}
\definecolor{darkgray}{RGB}{70,70,70}

\resizebox{\linewidth}{!}{%
\begin{tikzpicture}[
    font=\small,
    box/.style={
        draw=darkgray,
        rounded corners,
        align=center,
        minimum height=0.78cm,
        inner xsep=6pt,
        inner ysep=4pt,
        fill=boxgray
    },
    opbox/.style={box, fill=lightblue},
    krylovbox/.style={box, fill=lightorange},
    lossbox/.style={box, fill=lightgreen},
    gradbox/.style={box, fill=lightpurple},
    detachbox/.style={box, fill=orange!15},
    ad/.style={
        -latex,
        thick,
        draw=diffblue,
        shorten <=2pt,
        shorten >=2pt
    },
    nograd/.style={
        -latex,
        thick,
        dashed,
        draw=detachorange,
        shorten <=2pt,
        shorten >=2pt
    },
    lab/.style={font=\scriptsize}
]

% ============================================================
% Panel (a): direct AD through finite SLQ
% ============================================================
\node[anchor=west, text=diffblue] at (-0.1,1.1)
{\textbf{(a) Direct AD through the finite SLQ estimator}};

\node[opbox]     (g1)    at (0,0)      {$G$};
\node[opbox]     (p1)    at (1.7,0)    {$P(G)$};
\node[krylovbox] (lan1)  at (3.9,0)    {Lanczos\\recurrence};
\node[krylovbox] (tri1)  at (6.0,0)    {$Q_\ell,T_\ell$};
\node[krylovbox] (eig1)  at (8.0,0)    {eig$(T_\ell)$};
\node[krylovbox] (quad1) at (10.5,0)   {$\theta_{\ell j},\omega_{\ell j},v_{\ell j}$};
\node[lossbox]   (loss1) at (13.6,0)   {$\widehat{\mathcal L}_{m,N_{\mathrm p}}(G)$};
\node[gradbox]   (grad1) at (16.2,0)   {$\nabla_G\widehat{\mathcal L}$};

\draw[ad] (g1) -- (p1);
\draw[ad] (p1) -- (lan1);
\draw[ad] (lan1) -- (tri1);
\draw[ad] (tri1) -- (eig1);
\draw[ad] (eig1) -- (quad1);
\draw[ad] (quad1) -- (loss1);
\draw[ad] (loss1) -- (grad1);

\node[lab, text=darkgray, align=center] at (8.2,-0.85)
{The differentiated graph contains the Lanczos recurrence, Ritz values/vectors,\\
quadrature weights, and eigendecomposition of \(T_\ell\).};

% ============================================================
% Panel (b): detached Rayleigh support-gradient
% ============================================================
\node[anchor=west, text=detachorange] at (-0.1,-2.05)
{\textbf{(b) Detached Rayleigh support-gradient}};

% Detached forward spectral-data stage
\node[opbox]     (g0)    at (0,-3.15)     {$G_0$};
\node[opbox]     (p0)    at (1.7,-3.15)   {$P(G_0)$};
\node[krylovbox] (lan0)  at (3.9,-3.15)   {Lanczos\\computation};
\node[detachbox] (data0) at (6.6,-3.15)   {detach\\$v_{\ell j},\omega_{\ell j}$};

\draw[nograd] (g0) -- (p0);
\draw[nograd] (p0) -- (lan0);
\draw[nograd] (lan0) --
    node[above,lab,fill=white,inner sep=1pt,text=detachorange] {no AD}
    (data0);

\node[lab, text=detachorange, align=center] at (3.4,-4.0)
{Stage 1: compute spectral data once at \(G_0\), without gradient tracking.};

% Differentiated Rayleigh quotient stage
\node[opbox]   (g2)    at (0,-5.55)      {$G$};
\node[opbox]   (p2)    at (1.7,-5.55)    {$P(G)$};
\node[lossbox] (rho)   at (5.3,-5.55)    {$\rho_{\ell j}(G)
      =v_{\ell j}^{\top}P(G)v_{\ell j}$};
\node[lossbox] (loss2) at (9.0,-5.55)    {$\widetilde{\mathcal L}(G)$};
\node[gradbox] (grad2) at (13.0,-5.55)   {$\displaystyle
      \left.\frac{\partial \widetilde{\mathcal L}}
      {\partial G_{ik}}\right|_{(i,k)\in\mathcal S_G}$};

\draw[ad] (g2) -- (p2);
\draw[ad] (p2) -- (rho);
\draw[ad] (rho) -- (loss2);
\draw[ad] (loss2) --
    node[above,lab,fill=white,inner sep=1pt] {restrict to \(\mathcal S_G\)}
    (grad2);

\draw[nograd] (data0) --
    node[right,lab,fill=white,inner sep=1pt,text=detachorange]
    {fixed}
    (rho);

\node[lab, text=darkgray, align=center] at (7.2,-6.45)
{Stage 2: differentiate only the Rayleigh quotients with fixed \(v_{\ell j}\) and \(\omega_{\ell j}\).};

% Legend
\draw[ad] (14.4,-3.15) -- (15.1,-3.15);
\node[anchor=west,lab] at (15.2,-3.15) {differentiated};

\draw[nograd] (14.4,-3.55) -- (15.1,-3.55);
\node[anchor=west,lab] at (15.2,-3.55) {detached / no AD};

\end{tikzpicture}%
}
\caption{Computational graphs for direct differentiation of the finite SLQ
estimator and the detached Rayleigh support-gradient implementation.  In
panel (a), automatic differentiation is applied to the entire finite SLQ
pipeline, so the graph includes the Lanczos recurrence, the projected
tridiagonal matrix \(T_\ell\), the eigendecomposition of \(T_\ell\), and the
Ritz data used in the quadrature estimator.  In panel (b), the Lanczos
computation is first performed at the current factor \(G_0\) without gradient
tracking.  The resulting vectors and weights are detached and then treated as
fixed data.  The backward calculation differentiates only the recomputed
Rayleigh quotients \(v_{\ell j}^{\top}P(G)v_{\ell j}\), and the final gradient
is gathered on the prescribed support \(\mathcal S_G\).}
\label{fig:direct_vs_detached_lanczos}
\end{figure}

\subsection{Projected KPM backward}
\label{sec:projected_ad:kpm}

We also consider a projected Kernel Polynomial Method (KPM) backward rule
\cite{lin2016approximating,xi2018fast}.  In contrast to the
Lanczos--Rayleigh rule, which uses Ritz data to approximate
\(H(G)=n^{-1}f'(P(G))\), KPM first replaces the spectral function by a finite
Chebyshev polynomial and differentiates the resulting polynomial trace
estimator.

Let
\[
    B(G)=\frac{P(G)-cI}{a},
    \qquad a>0,
\]
with fixed scaling parameters \(a,c\).  If \(g(t)=f(c+at)\), let
\[
    p_M(t)=\sum_{k=0}^{M-1}\alpha_k T_k(t)
\]
be a degree-\((M-1)\) Chebyshev approximation, including any fixed damping
factors.  With probe vectors \(z_r\), the finite KPM estimator is
\begin{equation}
\label{eq:kpm_estimator}
    \widehat{\mathcal L}^{\mathrm{kpm}}_{M,R}(G)
    =
    \frac{1}{nR}\sum_{r=1}^R z_r^\top p_M(B(G))z_r .
\end{equation}

At the current factor \(G_0\), set \(B_0=B(G_0)\) and form
\[
    \mathcal Z_M
    =
    \operatorname{span}
    \{T_k(B_0)z_r:\;0\le k\le M-1,\;1\le r\le R\}.
\]
Let \(Z\) be an orthonormal basis for this space, computed without gradient
tracking and then held fixed.  Define
\[
    B_Z(G)=Z^\top B(G)Z,
    \qquad
    \zeta_r=Z^\top z_r .
\]
The projected estimator is
\begin{equation}
\label{eq:kpm_projected_estimator}
    \widetilde{\mathcal L}^{\mathrm{kpm}}_{M,R}(G)
    =
    \frac{1}{nR}\sum_{r=1}^R
    \zeta_r^\top
    p_M\!\left(B_Z(G)\right)
    \zeta_r .
\end{equation}
With fixed probes, fixed scaling, and the full projected basis \(Z\), the
projected KPM backward pass matches the gradient of the finite polynomial
estimator \eqref{eq:kpm_estimator} at the current point \(G_0\).  In other
words, after replacing \(f\) by the polynomial \(p_M\), the projected
computation does not introduce an additional local gradient approximation.

Let
\[
    S_Z
    =
    \operatorname{sym}\!\left(
        \nabla_{B_Z}
        \widetilde{\mathcal L}^{\mathrm{kpm}}_{M,R}(G)
    \right),
\]
where \(\operatorname{sym}(X)=(X+X^\top)/2\).  Since
\[
    B(G)=\frac{G^\top A G-cI}{a},
\]
the pullback from the projected matrix \(B_Z(G)\) to the factor is
\[
    \nabla_G
    \widetilde{\mathcal L}^{\mathrm{kpm}}_{M,R}(G)
    =
    \frac{2}{a} A G ZS_ZZ^\top ,
\]
restricted to \(\mathcal S_G\).  This is the \(P(G)=G^\top A G\) analogue of
the projected KPM backward rule.

The remaining approximation is the KPM approximation itself: the polynomial
estimator \eqref{eq:kpm_estimator} is still only an approximation to the 
spectral loss \(n^{-1}\operatorname{tr}(f(P(G)))\).  This distinction is
important for preconditioner construction.  A gradient can be exact for the
finite KPM objective and still be a poor direction for the original spectral
objective if the polynomial approximation, scaling, damping, or probe set
distort the spectral features that matter for Krylov convergence.  For the
bimodal loss, the most important region is near \(\lambda=0\), where the
logarithmic anti-zero term changes rapidly.  A global finite-degree Chebyshev
approximation can smooth this region, so optimizing the finite KPM objective
may not move the harmful near-zero eigenvalues as effectively as the
Lanczos--Rayleigh rule.
\section{Numerical Experiments}
\label{sec:experiments}

In this section, we evaluate the proposed approach on finite-element PDE test
problems.  We compare spectral and Frobenius objectives on fixed factor
supports, evaluate the proposed backward rule for matrix-free spectral
optimization, and assess a graph neural network model that predicts admissible factor entries across related matrices.

\subsection{Test matrices and experimental setup}
\label{sec:exp:setup}

We use three symmetric finite-element matrix families generated from standard
MFEM example problems~\cite{anderson2021mfem} on different meshes, retaining
the MFEM example names.  The \texttt{ex2} family is linear elasticity,
\[
    -\nabla\!\cdot\sigma(u)=f,
\]
on a two-dimensional triangular beam.  The \texttt{ex3} family is a definite
Maxwell problem,
\[
    \nabla\times(\mu^{-1}\nabla\times E)+\sigma E=f,
\]
on a two-dimensional square-disc.  The \texttt{ex5} family is a mixed
Darcy problem,
\[
    k u+\nabla p=f,\qquad -\nabla\!\cdot u=g,
\]
on a two-dimensional star-shaped domain.  The first two families are SPD,
whereas \texttt{ex5} is a symmetric indefinite saddle-point system.

The admissible support of \(G\) is fixed before optimizing its values.  Let
\[
    \mathcal S_p=\operatorname{triu}(\operatorname{pattern}(A^p)),
    \qquad p\in\{2,3\},
\]
denote the prescribed triangular structural-power support.  The FSAI baseline uses the Kolotilina--Yeremin columnwise construction on \(\mathcal S_p\).

We also include AINV as a nonsymmetric reference, using the biconjugation
procedure of Benzi et al.~\cite{benzi1998sparse}. 
We only tested the fixed-support version in which both triangular factors \(Z\) and \(W\) are
restricted to \(\mathcal S_p\).  
The AINV baselines therefore use approximately twice the storage of a single symmetric factor on \(\mathcal S_p\).  

Because AINV is nonsymmetric, all
preconditioners are evaluated with restarted flexible GMRES (FGMRES). 
We use restart \(50\), except for \texttt{ex2} where restart \(100\) is needed to avoid stagnation.  
Iteration counts are averaged over five standard-normal right-hand sides.  The relative tolerance is \(10^{-6}\), the maximum iteration count is \(2000\), and \texttt{fail} denotes failure of at least one right-hand side.

All computations use double precision, except the graph-network factor model of Section~\ref{sec:exp:factor_model}, which is trained and evaluated in single precision; the resulting preconditioner is applied within double-precision FGMRES as for every method.
The differentiable spectral losses and the matrix-specific spectral targets were implemented in PyTorch; the graph-network factor model was implemented in JAX/Flax. Motivated by prior evidence that preconditioned stochastic-gradient methods can improve optimization in scientific machine-learning problems~\cite{scott2026design}, we trained the graph-network factor model with the SOAP optimizer~\cite{vyas2024soap}.  MFEM was used only for offline matrix generation.
The experiments were run on an Ubuntu 24.04.4 LTS machine equipped with 64 GB of system memory, an Intel i7-12700K CPU, and an NVIDIA RTX 3070 Ti GPU. The main dependencies were PyTorch 2.9.1, JAX (with Flax and opttx), and CUDA 13.0. 

\subsection{Fixed-pattern spectral objectives}
\label{sec:exp:fixed_pattern}

The first set of experiments isolates the effect of the spectral loss directly,
without a neural network.  We fix the factor support and select the nonzero
entries by minimizing a loss with L-BFGS, so that the F-norm and spectral runs differ only in the objective used to choose the admissible values.  We compare the bimodal spectral loss \eqref{eq:bimodal_loss} against an algebraic baseline that minimizes the Frobenius residual
\[
    \|I-G^\top A G\|_F ,
\]
which we call the \emph{F-norm} baseline, and against the classical FSAI and AINV factors.

For the spectral objective, we use \(\alpha=0.1\) and
\(\varepsilon=10^{-8}\).  The analysis in Section~\ref{sec:loss} assumes
\(\delta>0\) for differentiability of the smoothed absolute value.  In the
small dense experiments of Table~\ref{tab:fixed_loss_ceiling_a2}, where the
full spectrum is available, we use the unsmoothed choice \(\delta=0\), i.e.,
the exact \(|\lambda|\); we label this exact-eigendecomposition variant
\emph{EigSpec}.  In the matrix-free differentiable experiments below,
the smoothed loss is used.  We use the bimodal loss uniformly across the three
families: on the positive spectrum of the SPD cases, it reduces to a
one-sided clustering penalty near \(+1\), together with an anti-zero
regularizer.  Thus the F-norm-versus-spectral comparison tests the same
objective family on every example.

The SPD problems (\texttt{ex2}, \texttt{ex3}) are initialized from the FSAI
factor.  The indefinite saddle-point problem \texttt{ex5}, where FSAI does not apply, is initialized from the identity. 
The optimizer uses a strong-Wolfe line search, a history size \(20\), and at most \(50\) outer steps, each consisting of up to \(20\) inner iterations. Optimization terminates early when the change in loss between outer steps falls below \(10^{-10}\).

\begin{table}[htbp]
\centering
\caption{Fixed-pattern objective comparison on the \(\mathcal S_2\) support.
The three test matrices are \texttt{ex2} with contrast
\(\lambda_0/\lambda_1=10\), \texttt{ex3} with \(\sigma=8\), and \texttt{ex5}
with \(k=4\). \textit{Fill} is the total
\(\operatorname{nnz}(Z)+\operatorname{nnz}(W)\) for AINV and the single-factor
\(\operatorname{nnz}(G)\) for the other methods; \textit{Iters} is
the mean FGMRES iteration count over the five right-hand sides.}
\label{tab:fixed_loss_ceiling_a2}
\small
\setlength{\tabcolsep}{5pt}
\begin{tabular}{lrrrrrr}
\toprule
 & \multicolumn{2}{c}{\texttt{ex2}} & \multicolumn{2}{c}{\texttt{ex3}} & \multicolumn{2}{c}{\texttt{ex5}} \\
\cmidrule(lr){2-3} \cmidrule(lr){4-5} \cmidrule(lr){6-7}
Method & Fill & Iters & Fill & Iters & Fill & Iters \\
\midrule
NoPrecond        & --- & \texttt{fail} & --- & 195.6 & --- & 178.8 \\
FSAI              &  4630 &  67.8 &  1277 &  71.0 &   --- & --- \\
F-norm            &  4630 &  70.4 &  1277 &  68.6 &  7590 & \texttt{fail} \\
EigSpec    &  4630 &  67.8 &  1277 &  66.6 &  7590 &   6.2 \\
AINV              &  8830 &  99.0 &  2324 & 167.4 & 15180 &  17.4 \\
\bottomrule
\end{tabular}
\end{table}

Table~\ref{tab:fixed_loss_ceiling_a2} reports the \(\mathcal S_2\)-support
comparison.  On the two SPD examples, spectral tuning is comparable to the classical single-factor baselines.  For \texttt{ex2}, FSAI and spectral tuning both require \(67.8\) iterations, while F-norm tuning requires
\(70.4\).  For \texttt{ex3}, spectral gives the lowest count among the
single-factor methods, \(66.6\) iterations, compared with \(68.6\) for F-norm and \(71.0\) for FSAI.  Thus, on these SPD cases, the spectral objective gives performance similar to or slightly better than the Frobenius-based and FSAI
baselines.

The indefinite saddle-point \texttt{ex5} case is qualitatively different.  FSAI is not reported for this indefinite problem, and F-norm tuning fails the true-residual check on the \(\mathcal S_2\) support. Spectral tuning, however, converges in \(6.2\) iterations on the same support with
\(\operatorname{nnz}(G)=7590\).  With the support fixed, the objective used to choose the admissible entries therefore determines whether the resulting
factor is an effective preconditioner for this indefinite problem.

The AINV is included as classical nonsymmetric references at roughly twice the fill of the single-factor methods.  
AINV require more iterations than spectral in this table, despite using larger total fill.  
The difference is especially pronounced on
\texttt{ex5}, where spectral reaches \(6.2\) iterations with a
single-factor fill of \(7590\), while AINV requires
\(17.4\) iterations with a total fill \(15180\).

\begin{figure}[htbp]
\centering
\input{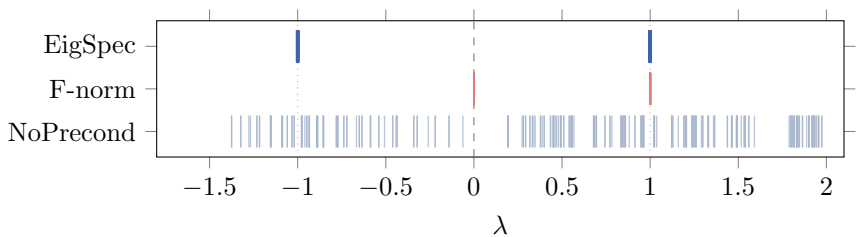}
\caption{Eigenvalue spectra for the \texttt{ex5} case (\(n=260\),
\(k=4\)) on the \(\mathcal S_2\) support.  Each tick is one eigenvalue of
the corresponding congruent operator \(G^\top A G\); the NoPrecond row shows
the eigenvalues of \(A\).  The F-norm and EicSpec rows use the same
support and differ only in the loss used to choose the factor entries.}
\label{fig:ex5_spectrum_before_after}
\end{figure}

Figure~\ref{fig:ex5_spectrum_before_after} illustrates the issue that motivated the introduction of the bimodal loss.
On the same \(\mathcal S_2\) support, the Frobenius
objective tries to fit the one-sided target
\[
    G^\top A G\approx I .
\]
For an invertible factor, this conflicts with Sylvester inertia, since the
congruent operator must retain the negative inertia of \(A\).  In the
optimized sparse factor, the negative branch is therefore driven toward small
magnitude rather than toward \(+1\), producing near-zero eigenvalues and the
observed stagnation.  The bimodal loss uses the attainable two-sided target
\[
    |\lambda|\approx 1,
\]
giving two clusters near \(\pm1\).

We next test the same conclusion across two mesh sizes, two support patterns, and \(k\in\{1,4,8\}\).  For this grid, we use the scalable matrix-free optimizer for all cases, including the \(n=260\) cases, so that the results are produced by a single optimization protocol.  The \emph{ProjSpec} column in Table~\ref{tab:ex5_formal_grid} is obtained with the detached Rayleigh rule of Section~\ref{sec:projected_ad:rayleigh}, rather than by eigendecomposition.  
The overlap with Table~\ref{tab:fixed_loss_ceiling_a2} at
\(n=260\), \(k=4\), and support \(\mathcal S_2\) gives \(10.0\) iterations
rather than \(6.2\).  This difference is expected: Table~\ref{tab:fixed_loss_ceiling_a2}
uses an accurate eigendecomposition to evaluate the spectral objective and
gradient, whereas Table~\ref{tab:ex5_formal_grid} uses the matrix-free
projected gradient.

\begin{table}[htbp]
\centering
\caption{Comparison on the \texttt{ex5} with varying mesh size, factor support, and \(k\).  Each entry is the mean FGMRES iteration count over the right-hand sides. The F-norm and ProjSpec columns use the prescribed support \(\mathcal S_p\).  \emph{AINV} uses prescribed
two-factor supports.}
\label{tab:ex5_formal_grid}
\scriptsize
\setlength{\tabcolsep}{4pt}
\begin{tabular}{llrrrrr}
\toprule
\(n\) & support & \(k\) & NoPrecond & F-norm & ProjSpec & AINV \\
\midrule
260 & \(\mathcal S_2\) & 1 & 222.4 & fail & 10.0 & 18.0 \\
    & \(\mathcal S_2\) & 4 & 178.8 & fail & 10.0 & 17.4 \\
    & \(\mathcal S_2\) & 8 & 202.8 & fail & 10.4 & 17.6 \\
\cmidrule(lr){2-7}
    & \(\mathcal S_3\) & 1 & 222.4 & fail &  6.0 &  7.8 \\
    & \(\mathcal S_3\) & 4 & 178.8 & fail &  6.0 &  7.8 \\
    & \(\mathcal S_3\) & 8 & 202.8 & fail &  6.0 &  8.0 \\
\midrule
1000 & \(\mathcal S_2\) & 1 & 578.8 & fail & 16.0 & 30.8 \\
     & \(\mathcal S_2\) & 4 & 483.2 & fail & 17.0 & 30.2 \\
     & \(\mathcal S_2\) & 8 & 609.4 & fail & 18.0 & 30.8 \\
\cmidrule(lr){2-7}
     & \(\mathcal S_3\) & 1 & 578.8 & fail & 10.0 & 12.0 \\
     & \(\mathcal S_3\) & 4 & 483.2 & fail & 10.8 & 12.0 \\
     & \(\mathcal S_3\) & 8 & 609.4 & fail & 11.8 & 12.0 \\
\bottomrule
\end{tabular}
\end{table}

Table~\ref{tab:ex5_formal_grid} reports the resulting \texttt{ex5} grid.
Across all twelve cases, the projected spectral factors reduce the
unpreconditioned counts from \(178.8\)--\(609.4\) iterations to
\(6.0\)--\(18.0\).  F-norm tuning fails the true-residual test on every
saddle-point case.  The prescribed-pattern \emph{AINV} baseline uses roughly twice the storage of
the single spectral factor and takes \(7.8\)--\(30.8\) iterations; it requires
more iterations than the projected spectral factor in every cell.

The advantage of the projected spectral factor is largest on the smaller
\(\mathcal S_2\) support, especially on the \(n=1000\) mesh, where projected
spectral takes \(16.0\)--\(18.0\) iterations compared with \(30.2\)--\(30.8\)
for \emph{AINV}.  On the larger \(\mathcal S_3\) support, the gap narrows, but
projected spectral still gives lower iteration counts in every tested case:
\(6.0\) versus \(7.8\)--\(8.0\) on the \(n=260\) mesh and \(10.0\)--\(11.8\)
versus \(12.0\) on the \(n=1000\) mesh.  Increasing the support from
\(\mathcal S_2\) to \(\mathcal S_3\) improves both baselines, but the projected
spectral factor remains competitive or better while using a single-factor
support rather than the two-factor \emph{AINV} support.

\subsection{Backward-rule ablation for the SLQ estimator}
\label{sec:exp:backward_ablation}

We next examine the local gradient behavior of the matrix-free SLQ rule.  The
finite SLQ estimator \eqref{eq:slq_estimator} gives a stochastic approximation
to the spectral loss.  
A direct backward pass differentiates through the entire finite Lanczos process, whereas the detached Rayleigh rule of Section~\ref{sec:projected_ad:rayleigh} holds the Ritz data fixed, returning a gradient on the prescribed support. To asses the accuracy of these gradients, we compare both backward rules with the support restricted reference
\[
    \Pi_{\mathcal S_p}\!\left(
        \frac{2}{n} A G f'(P(G))
    \right),
    \qquad
    P(G)=G^\top A G ,
\]
which is the exact eigendecomposition-based spectral gradient restricted to the active entries of \(G\).

The ablation is performed on the saddle-point \texttt{ex5} family at three
mesh refinements, \(n=260\), \(1000\), and \(3920\).  We evaluate all gradients
at \(G_0=I\) over the 50 log-uniform Darcy coefficients
\(k\in[0.5,16]\).  The two SLQ backward routes use the same stochastic state and Lanczos settings:
\(64\) Rademacher probes, \(m=40\) Lanczos steps, and one reorthogonalization pass.

\begin{table}[htbp]
\centering
\caption{Backward-rule comparison for the SLQ estimator
\eqref{eq:slq_estimator} on \texttt{ex5} at
\(G_0=I\).  \emph{Direct SLQ AD} differentiates through the finite Lanczos
process; \emph{Detached Rayleigh} is the rule of
Section~\ref{sec:projected_ad:rayleigh}.
\(\cos(\nabla,\nabla_{\mathrm{dense}})\) is the cosine of each estimator's gradient with the accurate reference.
Results are reported as mean \(\pm\) standard deviation over the 50 values of $k$; peak MB is the mean peak GPU allocation.}
\label{tab:backward_ablation_logk}
\small
\setlength{\tabcolsep}{4pt}
\begin{tabular}{llllll}
\toprule
\(n\) & support & Method & \(\cos(\nabla,\nabla_{\mathrm{dense}})\) & step (s) & peak MB \\
\midrule
260  & \(\mathcal S_2\) & Direct SLQ AD & \(0.61\pm0.33\) & \(0.230\pm0.014\) & \(162\) \\
     &         & Detached Rayleigh & \(0.95\pm0.15\) & \(0.091\pm0.003\) & \( 62\) \\
     & \(\mathcal S_3\) & Direct SLQ AD & \(0.53\pm0.31\) & \(0.238\pm0.003\) & \(234\) \\
     &         & Detached Rayleigh & \(0.93\pm0.15\) & \(0.096\pm0.003\) & \(124\) \\
\midrule
1000 & \(\mathcal S_2\) & Direct SLQ AD & \(0.74\pm0.22\) & \(0.279\pm0.013\) & \(598\) \\
     &         & Detached Rayleigh & \(0.97\pm0.03\) & \(0.119\pm0.003\) & \(215\) \\
     & \(\mathcal S_3\) & Direct SLQ AD & \(0.63\pm0.21\) & \(0.320\pm0.002\) & \(1020\) \\
     &         & Detached Rayleigh & \(0.93\pm0.05\) & \(0.128\pm0.003\) & \(582\) \\
\midrule
3920 & \(\mathcal S_2\) & Direct SLQ AD & \(0.79\pm0.15\) & \(0.416\pm0.012\) & \(2427\) \\
     &         & Detached Rayleigh & \(0.97\pm0.02\) & \(0.157\pm0.003\) & \(828\) \\
     & \(\mathcal S_3\) & Direct SLQ AD & \(0.67\pm0.16\) & \(0.617\pm0.003\) & \(4437\) \\
     &         & Detached Rayleigh & \(0.93\pm0.03\) & \(0.192\pm0.002\) & \(2584\) \\
\bottomrule
\end{tabular}
\end{table}

Table~\ref{tab:backward_ablation_logk} reports the results.  The detached
Rayleigh rule attains mean gradient cosines between \(0.93\) and \(0.97\)
against the dense support-restricted reference across all refinements and
supports.  Direct SLQ automatic differentiation gives lower mean cosines,
between \(0.53\) and \(0.79\), and has substantially larger variation over the
50 values of \(k\).  The detached rule is also about \(2.3\)--\(3.2\times\)
faster per forward--backward call and reduces peak GPU allocation by factors
of about \(1.7\)--\(2.9\).

The dense reference becomes expensive at the largest refinement.  At
\(n=3920\), one dense forward--backward pass takes \(3.47\) s and allocates
\(1.5\)--\(3.2\) GB of peak GPU memory.  The detached Rayleigh rule takes
\(0.16\)--\(0.19\) s and uses \(0.83\)--\(2.58\) GB on the same inputs.  This
is the first tested refinement where the matrix-free projected gradient is
strictly faster than the dense oracle, with lower or comparable peak memory in
the tested cells.

It is worth emphasizing that the detached rule is not the derivative of the finite SLQ estimator.  Direct
SLQ automatic differentiation differentiates the Krylov basis, Ritz values,
quadrature weights, and the eigendecomposition of \(T_\ell\).  The detached
rule omits these algorithmic derivatives and returns the projected
support-gradient of Section~\ref{sec:projected_ad}.  On these tests, this
projected gradient is biased relative to the finite SLQ derivative but better
aligned with the support-restricted dense spectral gradient we target.  It is
both cheaper and more accurate in this local diagnostic, so we use it in the
optimization experiments below.  This is a local test at \(G_0=I\); it does
not characterize the full optimization trajectory.

\subsection{Graph-network factor model}
\label{sec:exp:factor_model}

We next study whether neural network model can be used as a
parameterized preconditioner for a family of related sparse systems.  
In parameter-dependent PDE discretizations, inverse problems, and multilevel or multifidelity computations, one often needs preconditioners for many matrices
that differ in mesh size, physical parameter, or prescribed factor support.
Optimizing a separate spectral factor for every such matrix would defeat the
purpose of a learned preconditioner.  Instead, the goal is to learn a single
parameterized factor model that maps local matrix and problem features to the
admissible entries of \(G\).

We test this idea on the \texttt{ex5} Darcy family.  The controlled axes are
mesh size, prescribed support pattern, and the Darcy parameter \(k\).  We
separate the experiment along these axes.  The first test uses one set of model
weights across two mesh sizes with fixed \(k\) and fixed support.  The second
uses one set of weights across the two prescribed supports
\(\mathcal S_2\) and \(\mathcal S_3\) with fixed mesh and fixed \(k\).  The
third trains on several values of \(k\) and evaluates interpolation to
held-out parameter values.  Together, these tests ask whether one
parameterized factor model can produce effective fixed-support preconditioners
across the main variations in this parametric matrix family, without
reoptimizing a separate spectral factor for each new case.

Following the graph-network framework of
\cite{battaglia2018relational}, the model takes matrix and problem features as
input and predicts the admissible entries of \(G\) on the prescribed support
\(\mathcal S_p\).  The graph nodes are the degrees of freedom of the
finite-element matrix.  In addition, the model maintains one factor-entry token
for each admissible entry \((i,j)\in\mathcal S_p\).  The scalar outputs are
indexed by these factor-support entries; hence, for a fixed support, the model
parameterizes a sparse factor with exactly the prescribed nonzero pattern.

Each graph-network block maintains node embeddings \(x_i\) and factor-entry
embeddings \(e_{ij}\).  The entry update uses the current \(e_{ij}\) together
with the embeddings of its two endpoint nodes.  The node update aggregates two
types of messages: messages from incident factor-support entries and messages
propagated along the matrix-adjacency graph.  A virtual node enables global communication across the graph.

The final entry predictor outputs a correction to the identity factor rather
than predicting the factor entry from scratch.  For an admissible entry
\((i,j)\in\mathcal S_p\), the predicted value is
\[
    \widehat g_{ij}
    =
    \mathbf{1}_{\{i=j\}}
    + h_{\mathrm{edge}}(e_{ij})
    + \eta\, r^{-1/2} u_i^\top v_j ,
\]
where \(h_{\mathrm{edge}}\) is a local edge head,
\(u_i,v_j\in\mathbb R^r\) are endpoint embeddings computed from the final node
states, \(r=128\), and \(\eta\) is a learned scalar gain.  Thus diagonal
entries are initialized from the identity, while off-diagonal entries start
from zero and are learned as corrections on the prescribed support.

The input features are local sparse-matrix and problem features, not spectral
features.  Node features include spatial coordinates, Fourier coordinate
features, a variable-type indicator, a virtual-node flag, and, when present,
the PDE parameter.  For each factor-support entry
\((i,j)\in\mathcal S_p\), the edge features include \(a_{ij}\), \(a_{ii}\),
\(a_{jj}\), diagonal and block indicators, and sparse matrix-polynomial
features from \(A^2\) and \(A^3\) restricted to \(\mathcal S_p\).  We use
simple scalings, nonzero indicators, and signed-log magnitudes of these scalar
quantities.  In the tests below, the architecture is kept fixed: three
graph-network blocks, hidden width \(256\), low-rank dimension \(128\), Fourier
coordinate features, matrix-polynomial edge features, and a virtual node.

For each case used in supervised training, the target factor is obtained by
matrix-specific spectral optimization on the same prescribed support.  The
graph-network model is then trained with a relative \(L^2\) loss over the
admissible entries of \(G\), using the SOAP optimizer in single precision with
a three-stage learning-rate decay
(\(10^{-3}\!\to\!3\times10^{-4}\!\to\!10^{-4}\) over the first \(30/30/40\%\) of
training); we report the final iterate, without checkpoint selection.  We
evaluate the learned factor by applying it as a preconditioner under the FGMRES
protocol of Section~\ref{sec:exp:setup}.

We first isolate the two structural axes of the factor family.  In the
mesh-size test, one model is trained jointly on the two mesh sizes
\(n=260,1000\), with \(k=4\) and support \(\mathcal S_2\).  In the
support-pattern test, one model is trained jointly on the two supports
\(\mathcal S_2,\mathcal S_3\), with \(n=260\) and \(k=4\).

\begin{table}[htbp]
\centering
\caption{Controlled sharing tests for the graph-network model on
\texttt{ex5}.  Entries are mean FGMRES iteration counts.  \emph{Target}
denotes the matrix-specific spectral target used for supervised training.
Left: one model trained on two mesh sizes with \(k=4\) and support
\(\mathcal S_2\).  Right: one model trained on two supports with
\(n=260\) and \(k=4\).}
\label{tab:gn_controlled_sharing}
\small
\begin{tabular}{@{}lrrr@{\qquad}lrrr@{}}
\toprule
\multicolumn{4}{c}{Mesh-size sharing}
&
\multicolumn{4}{c}{Support-pattern sharing} \\
\cmidrule(lr){1-4}\cmidrule(lr){5-8}
\(n\) & AINV & Target & GN
&
support & AINV & Target & GN \\
\midrule
260  & 17.4 & 10.0 & 10.0
&
\(\mathcal S_2\) & 17.4 & 10.0 & 10.0 \\
1000 & 30.2 & 17.0 & 17.4
&
\(\mathcal S_3\) &  7.8 &  6.0 &  6.0 \\
\bottomrule
\end{tabular}
\end{table}

Table~\ref{tab:gn_controlled_sharing} reports the two controlled sharing
tests.  In the mesh-size test, one graph-network model essentially matches the
matrix-specific spectral target on both refinements: \(10.0\) iterations at
\(n=260\) and \(17.4\) iterations at \(n=1000\), against targets of \(10.0\)
and \(17.0\).  The GN factors also improve substantially on AINV, which takes
\(17.4\) and \(30.2\) iterations on the two meshes, respectively.

In the support-pattern test, one model trained jointly on \(\mathcal S_2\)
and \(\mathcal S_3\) again matches the matrix-specific spectral target.  It
requires \(10.0\) iterations on \(\mathcal S_2\) and \(6.0\) iterations on
\(\mathcal S_3\), compared with \(17.4\) and \(7.8\) iterations for AINV.
This test changes the number and locations of admissible factor entries, so
the output index set is different across the two supports.  The result shows
that the same graph-network architecture and weights can be used across these
prescribed supports.

The remaining axis is the PDE parameter.  For this test, we fix the small
\texttt{ex5} mesh and the \(\mathcal S_2\) support, and train on \(128\)
log-uniform values of \(k\in[3,5]\).  We evaluate on \(32\) randomly chosen
held-out values of \(k\).  The held-out values are not used during
supervised training, and no spectral fine-tuning is applied.  For each
training value of \(k\), the supervised target is the corresponding
matrix-specific spectral factor on \(\mathcal S_2\).  Matrix-specific spectral
targets are also computed on the held-out values only as references for
evaluation.  We use the same graph-network architecture as above and train for
\(1000\) epochs (\(4000\) optimizer steps) with minibatches of \(32\) cases.

On the held-out values of \(k\), the supervised graph-network model requires
\(10.21\) FGMRES iterations on average, compared with \(17.46\) for AINV on the
same prescribed support, and it improves on AINV at every one of the \(32\)
held-out values.  The matrix-specific spectral target requires \(10.0\)
iterations on average.  Thus, on this restricted parameter range, the learned
factor nearly matches the matrix-specific target and clearly improves on AINV,
without any held-out fine-tuning.
\section{Conclusion}
\label{sec:conclusion}

We studied fixed-support factorized sparse approximate inverse
preconditioners whose admissible entries are chosen by spectral objectives of
the congruent operator $P(G)=G^\top A G$.  The main value-selection
contribution is a bimodal loss for symmetric indefinite systems.  This loss
respects inertia by driving eigenvalues away from zero and toward separated
positive and negative clusters.  The experiments show that this spectral
criterion can make a prescribed sparse support effective on saddle-point
problems where Frobenius-based tuning leaves nearly singular preconditioned
operators.

To make these spectral objectives practical, we developed matrix-free
projected Krylov support-gradients.  The SLQ-based rule uses Lanczos Ritz data
to approximate the dense spectral derivative \(n^{-1}f'(P(G))\), and evaluates
the resulting gradient only on the prescribed factor support.  Its detached
Rayleigh implementation avoids differentiating through the Lanczos recurrence:
the Lanczos data define local spectral information, while the backward pass
differentiates only recomputed Rayleigh quotients.  The projected KPM rule
provides a finite-polynomial comparison.

We also introduced a parameterized node-edge factor model for predicting
admissible factor entries across related matrices.  The model uses
entry-indexed states for the factor entries together with auxiliary states on
the original matrix unknowns.  The experiments show that entrywise agreement
with optimized factors is not by itself a sufficient proxy for solver quality,
and that spectral fine-tuning can further reduce Krylov iteration counts on
held-out parameter values.

Future work includes adaptive support selection, multilevel or Schur-aware
factor models for saddle-point systems, stronger generalization tests across
matrix families and mesh resolutions, and extensions of the projected
support-gradient framework to nonsymmetric preconditioned operators.

\bibliographystyle{siamplain}
\bibliography{papers}
\end{document}